\documentclass{amsart}
\address{Department of Mathematics, Massachusetts Institute of Technology, Cambridge, MA 02139, USA}
\email{ryba@mit.edu}
\title{An Answer to a Question of Zeilberger and Zeilberger about Fractional Counting of Partitions}
\author{Christopher Ryba}

\date{\today}

\usepackage{hyperref}
\usepackage{fullpage}
\usepackage{amsmath}
\usepackage{amsthm}
\usepackage{amsfonts}
\usepackage{amssymb}
\usepackage{ytableau}
\usepackage{subfig}
\usepackage{graphicx, caption, float}
\usepackage{verbatim}
\usepackage{todonotes}
\usepackage{ytableau}
\usepackage{esvect}
\begin{document}
\maketitle

\newtheorem{theorem}{Theorem}[section]
\newtheorem{lemma}[theorem]{Lemma}
\newtheorem{proposition}[theorem]{Proposition}
\newtheorem{corollary}[theorem]{Corollary}
\newtheorem{definition}[theorem]{Definition}
\newtheorem{remark}[theorem]{Remark}
\newtheorem{example}[theorem]{Example}

\newcommand\numberthis{\addtocounter{equation}{1}\tag{\theequation}}

\begin{abstract}
We answer a question of Zeilberger and Zeilberger about certain partition statistics.
\end{abstract}

\section{Introduction}
\noindent
For a partition $\lambda = (\lambda_1, \lambda_2, \ldots, \lambda_l)$, define $w_\lambda = \lambda_1 \lambda_2 \cdots \lambda_l$ (this is the product of the parts of $\lambda$). Zeilberger and Zeilberger \cite{ZZ} define two quantities:
\[
b(n) = \sum_{\lambda \vdash n} \frac{1}{w_\lambda}.
\]
and
\[
b(n, k) = \sum_{\substack{\lambda \vdash n \\ \lambda_1 = k}} \frac{1}{w_\lambda}.
\]
The latter sum is over partitions of $n$ whose largest part is equal to $k$, so $b(n) = \sum_{i=1}^n b(n,k)$. They ask to determine
\[
f(x) = \lim_{n \to \infty} b(n, \lfloor xn \rfloor)
\]
as a function on $[0,1]$. To answer this question, we use two tools. Firstly, a recurrence for $b(n,k)$ given by
Zeilberger and Zeilberger \cite{ZZ}:
\[
b(n,k) = \frac{1}{k} \sum_{i=1}^k b(n-k, i).
\]
Secondly, we use the asymptotic behaviour of $b(n)$, first considered by Lehmer \cite{Lehmer}.
\begin{theorem}[Lehmer]
We have $b(n) = e^{-\gamma}n(1 + o(1))$ as $n \to \infty$, where $\gamma$ is Euler's gamma.
\end{theorem}

\subsection{Acknowledgements}
The author would like to thank Andrew Ahn and Pavel Etingof for useful conversations.

\section{Understanding $b(n,k)$}
\noindent
In this section $x$ will be a number in $[0,1]$.
\begin{definition}
Let
\[
c(n,k) = e^{\gamma} b(n, k)
\]
and
\[
c(n) = e^{\gamma} b(n).
\]
\end{definition}
\noindent
Using this new function will make the following calculations cleaner. For example, $\lim_{n \to \infty} c(n)/n = 1$ according to our new convention. Note that $c(n,k)$ satisfies the same recurrence identities as $b(n, k)$.
\begin{example}
Suppose that $x \in (1/2, 1]$. Then for $n$ sufficiently large, we have
\[
c(n, \lfloor xn \rfloor) = \frac{1}{\lfloor xn \rfloor} \sum_{i=1}^{\lfloor xn \rfloor} c(n - \lfloor xn \rfloor, i) = \frac{c(n - \lfloor xn \rfloor)}{\lfloor xn \rfloor},
\]
because $\lfloor xn \rfloor \geq n - \lfloor xn \rfloor$ for $n$ sufficiently large. By Theorem 1.1, we may take the limit as $n \to \infty$, and obtain $\frac{1-x}{x}$.
\end{example}
\begin{proposition}
For $r \in \mathbb{Z}_{>0}$, there exists a smooth function $F_r(t)$ such that for $x \in (\frac{1}{r+1}, \frac{1}{r}]$,
\[
c(n, \lfloor xn \rfloor ) = F_r(x) + o(1)
\]
as $n \to \infty$. Moreover, these $F_r(x)$ are related via
\[
F_r(x) = \frac{1-x}{x} - \frac{1-x}{x}\left( \int_{\frac{x}{1-x}}^{\frac{1}{r-1}} F_{r-1}(t) dt + \sum_{s=1}^{r-2} \int_{\frac{1}{s+1}}^{\frac{1}{s}} F_s(t) dt \right).
\]
\end{proposition}
\begin{proof}
Example 2.2 demonstrated this for $x \in (1/2, 1]$, where we obtained $F_1(x) = \frac{1-x}{x}$; this forms the base case of an induction on $r$. We now assume $x \in (\frac{1}{r+1}, \frac{1}{r}]$;
\[
c(n, \lfloor xn \rfloor ) = \frac{1}{\lfloor xn \rfloor} \sum_{i=1}^{\lfloor xn \rfloor} c(n - \lfloor xn \rfloor , i) = \frac{1}{\lfloor xn \rfloor}\left(c(n - \lfloor xn \rfloor) - \sum_{i=\lfloor xn \rfloor +1}^{n - \lfloor xn \rfloor} c(n - \lfloor xn \rfloor, i)\right).
\]
In the latter sum, the ratio $\frac{i}{n - \lfloor xn \rfloor}$ is minimised when $i = \lfloor xn \rfloor +1$, and the resulting quantity is a weakly decreasing function of $x$. Because $x > \frac{1}{r+1}$, we conclude
\[
\frac{i}{n- \lfloor xn \rfloor} \geq \frac{\lfloor \frac{n}{r+1} \rfloor + 1}{n - \lfloor \frac{n}{r+1} \rfloor} \geq 1/r.
\]
We may therefore apply the induction hypothesis to the terms in the sum.
\begin{eqnarray*}
c(n, \lfloor xn \rfloor) &=&  \frac{1}{\lfloor xn \rfloor}\left(c(n - \lfloor xn \rfloor) - \left(\sum_{i= \lfloor xn \rfloor+1}^{\lfloor \frac{n - \lfloor xn \rfloor}{r-1}\rfloor} F_{r-1}\left(\frac{i}{n - \lfloor xn \rfloor}\right) +o(1) \right. \right.\\
&=&\left. \left. \hspace{10mm} +  \sum_{s = 1}^{r-2} \sum_{i= \lfloor \frac{n - \lfloor xn \rfloor}{(s+1)}\rfloor+1}^{ \lfloor \frac{n - \lfloor xn \rfloor}{s}\rfloor} F_s\left(\frac{i}{n - \lfloor xn \rfloor}\right) + o(1) \right)\right)
\end{eqnarray*}
Each term is a Riemann sum converging to an integral of the corresponding $F_s$. We note that although each $o(1)$ error term is summed $\mathcal{O}(n)$ times, this is accounted for by the leading factor of $1/\lfloor xn \rfloor$, so these still vanish in the limit $n \to \infty$. Note that we have
\begin{eqnarray*}
\lim_{n \to \infty} \frac{1}{n} \sum_{i= \lfloor \frac{n - \lfloor xn \rfloor}{(s+1)}\rfloor+1}^{\lfloor \frac{n - \lfloor xn \rfloor}{s}\rfloor} F_s\left(\frac{i}{n - \lfloor xn \rfloor}\right)
&=&
\int_{\frac{1-x}{s+1}}^{\frac{1-x}{s}} F_s\left(\frac{t}{1-x}\right) dt 
=
(1-x)\int_{\frac{1}{s+1}}^{\frac{1}{s}} F_s(t) dt.
\end{eqnarray*}
We conclude that
\[
\lim_{n \to \infty} c(n, \lfloor xn \rfloor) = \frac{1-x}{x} - \frac{1-x}{x}\left( \int_{\frac{x}{1-x}}^{\frac{1}{r-1}} F_{r-1}(t) dt + \sum_{s=1}^{r-2} \int_{\frac{1}{s+1}}^{\frac{1}{s}} F_s(t) dt \right).
\]
For $x \in (\frac{1}{r+1}, \frac{1}{r}]$, it is this quantity which we define to be $F_{r}(x)$, and the above limit is exactly the statement of the proposition. We conclude that $\lim_{n \to \infty} c(n, \lfloor nx \rfloor)$ is smooth for $x \notin \{1/n \mid n \in \mathbb{Z}_{>0}\}$. 
\end{proof}
\begin{example}
We may compute
\[
F_2(x) =  \frac{1-x}{x} - \frac{1-x}{x}\left( \int_{\frac{x}{1-x}}^{1} \frac{1-t}{t} dt \right) = \frac{2-3x}{x} - \frac{1-x}{x} \log\left( \frac{1-x}{x} \right).
\]
\end{example}
\begin{remark}
We may differentiate the expression for $F_r(x)$ to obtain a differential equation satisfied by $F_r(x)$:
\[
\frac{d}{dx} \left( \frac{x}{1-x} F_{r}(x) \right)  = \frac{1}{(1-x)^2}F_{r-1}\left( \frac{x}{1-x} \right)
\]
\end{remark}
\noindent
Finally, we obtain our result.
\begin{corollary}
Because $c(n,k)$ and $b(n,k)$ differed only by rescaling, and the above relations are linear in the $F_r$, we have
\[
\lim_{n \to \infty} b(n, \lfloor xn \rfloor) = e^{-\gamma}F_r(x)
\]
whenever $ x \in (\frac{1}{r+1}, \frac{1}{r}]$.
\end{corollary}
\begin{remark}
Suppose we assemble all the functions $F_r(x)$ into a single function $F(x)$ on $(0,1]$ (and say $F(x)=0$ for $x>1$). Let $G(x) = F(1/x)$. Then, the differential equation becomes
\[
G(x) - (x-1)G^\prime(x) = G(x-1).
\]
The upshot of this is that the current equation is well adapted for a Laplace transform. Writing $\hat{G}(t)$ for the Laplace transform of $G(x)$, we obtain:
\[
\hat{G}(t)+ (t \hat{G}(t) - G(0)) + \frac{d}{dt}(t\hat{G}(t) - G(0)) = e^{-t}\hat{G}(t),
\]
using the boundary condition $G(0) = 0$, this becomes
\[
\frac{d}{dt} \hat{G}(t) = \frac{e^{-t}-t-2}{t} \hat{G}(t).
\]
We may solve this explicitly:
\[
\hat{G}(t) = K t^{-2}\exp(Ei(-t) - t),
\]
where $Ei$ is the exponential integral, and $K$ is a constant.
\end{remark}

\bibliographystyle{alpha}
\bibliography{mybib.bib}

\end{document}